\documentclass[a4paper,11pt]{article}
\usepackage[utf8]{inputenc}
\usepackage{amsmath}

\usepackage[margin=0.9in]{geometry}

\newtheorem{theorem}{Theorem}[section]

\newtheorem{corollary}[theorem]{Corollary}

\newenvironment{proof}[1][Proof]{\begin{trivlist}
\item[\hskip \labelsep {\bfseries #1}]}{\end{trivlist}}

\newcommand{\qed}{\nobreak \ifvmode \relax \else
      \ifdim\lastskip<1.5em \hskip-\lastskip
      \hskip1.5em plus0em minus0.5em \fi \nobreak
      \vrule height0.75em width0.5em depth0.25em\fi}

\title{Positivity preserving DG schemes for a Boltzmann - Poisson model of electrons in semiconductors in curvilinear momentum coordinates}
\author{Jos\'e A. Morales E.$^1$, Irene M. Gamba$^2$, Eirik Endeve$^3$, Cory Hauck$^3$}
\date{$^1$ TU Wien - Institute for Analysis and Scientific Computing\\
$^2$The University of Texas at Austin - Institute for Computational Engineering and Sciences \& Department of Mathematics\\
$^3$Oak Ridge National Lab - Computational and Applied Mathematics Group
}

\begin{document}

\maketitle

\begin{abstract}

The work presented in this paper is related to the
development of positivity preserving Discontinuous Galerkin (DG) methods
for Boltzmann - Poisson (BP) computational models of electronic transport in semiconductors. 
We pose the Boltzmann Equation for electron transport in
curvilinear coordinates for the momentum. 
We consider the 1D diode problem with azimuthal symmetry, which is a 3D plus time problem. 
We choose for this problem the spherical coordinate system $\vec{p}(|\vec{p}|,\mu=cos\theta,\varphi)$, slightly different to the choice in previous DG solvers for BP, because its DG formulation gives simpler integrals 
involving just piecewise polynomial functions for both transport and collision terms. Applying the strategy of Zhang \& Shu, \cite{ZhangShu1}, \cite{ZhangShu2}, Cheng, Gamba, Proft, \cite{CGP}, and Endeve et al. \cite{EECHXM-JCP}, we treat the collision operator as a source term, and find convex combinations of the transport and collision terms which guarantee the positivity of the cell average of our numerical probability density function at the next time step. The positivity of the numerical solution to the pdf in the whole domain is guaranteed by applying the limiters in \cite{ZhangShu1}, \cite{ZhangShu2} that preserve the cell average but modify the slope of the piecewise linear solutions in order to make the function non - negative. In addition of the proofs of positivity preservation in the DG scheme, we prove the stability of the semi-discrete DG scheme under an entropy norm, using the dissipative properties of our collisional operator given by 
its entropy inequalities. The entropy inequality we use depends on an exponential of the Hamiltonian rather than the Maxwellian associated just to the kinetic energy.

\end{abstract}

\section{Introduction: Boltzmann Equation with Momentum in Curvilinear Coordinates}

We can write the Boltzmann - Poisson model for electron transport in semiconductors for a more general set of collision operators as the system in the $(\vec{x},\vec{p})$ position-momentum phase space for electrons
\begin{equation}
\partial_t f +  \partial_{\vec{x}} f \cdot \partial_{\vec{p}} \varepsilon
+  \partial_{\vec{p}} f \cdot q \partial_x V
= Q(f) 
= \int_{\Omega_{\vec{p}}} S(\vec{p}\,' \rightarrow \vec{p} ) f' d\vec{p}\,'
- f \int_{\Omega_{\vec{p}}} S(\vec{p} \rightarrow \vec{p}\,' ) d\vec{p}\,'
\end{equation}
\begin{equation}
-\partial_{\vec{x}} \cdot (\epsilon \partial_{\vec{x}} V)(\vec{x},t) =
q \left[ N(\vec{x}) - \int_{\Omega_{\vec{p}}} f(\vec{x},\vec{p},t) d\vec{p} \right], \quad \vec{E}(\vec{x},t) = -\partial_{\vec{x}}V(\vec{x},t)
\end{equation}
The momentum variable is $\vec{p}=\hbar\vec{k}$,
$\vec{k}$ is the crystal momentum wave vector,
$\varepsilon(\vec{p})$ is the conduction energy band for electrons in the semiconductor,
$f(\vec{x},\vec{p},t)$ is the probability density function (pdf) in the phase space for electrons in the conduction band,
$\vec{v}(\vec{p}) = \partial_{\vec{p}} \varepsilon (\vec{p}) $
is the quantum mechanical electron group velocity, $q$ is the positive electric charge,   
 $V(\vec{x},t)$ is the electric potential (we assume that the only force over the electrons is the self-consistent electric field, and that it is given by the negative gradient of the electric potential), $\epsilon$ is the permittivity for the material,  $N(\vec{x})$ is the fixed doping background in the semiconductor material, and $S(\vec{p}\,' \rightarrow \vec{p} )$ is the scattering kernel that defines the gain and loss operators, whose difference give the collision integral $Q(f)$.\\

For many collision mechanisms in semiconductors,
the scattering kernel $S(\vec{p}\,' \rightarrow \vec{p} )$ 
depends on the difference
$\varepsilon(\vec{p}) - \varepsilon(\vec{p}')$, as
in collision operators of the form  
$\delta( \varepsilon(\vec{p}) - \varepsilon(\vec{p}') + l\hbar w_p)$
for electron - phonon collisions. This form is related to energy conservation equations such as Planck's law, in which the jump in energy from one state to another is balanced with the energy of a phonon. The mathematical consequence of this is that 
we can obtain much simpler expressions for the integration of the collision operator if we express the momentum in curvilinear coordinates that involve the energy $\varepsilon(\vec{p})$ as one of the variables \cite{MP}, \cite{carr03},  \cite{cgms06}, \cite{CGMS-CMAME2008} . The other two momentum coordinates could be either an orthogonal system in the level set of energies, orthogonal to the energy gradient itself, or angular coordinates which are known to be orthogonal to the energy in the limit of low energies close to a local conduction band minimum, such as $(\mu, \varphi)$.

This gives both physical and mathematical motivations to pose the
Boltzmann Equation for semiconductors in curvilinear coordinates for the momentum $\vec{k}(k_1,k_2,k_3)$, to later on choose
the particular case of curvilinear coordinates such as $(\varepsilon,\mu,\varphi)$. 
We will assume in the rest of this chapter that our
system of curvilinear coordinates for the momentum 
 is orthogonal, as in the case $(\varepsilon,\mu,\varphi)$ in
which $\varepsilon(|\vec{p}|)$ is a monotone increasing function,
so this set of coordinates is equivalent to the representation in spherical coordinates for the momentum.  

The Boltzmann Equation for semiconductors (or more general forms of linear collisional plasma models) written in orthogonal curvilinear coordinates $\vec{p}(p_1,p_2,p_3)$ for the momentum $\vec{p} = (p_x,p_y,p_z)$
is

\begin{equation}
\partial_t Jf +  \partial_{\vec{x}}   \cdot \left( J f  
\vec{v} \right)
+ q  \left[
\partial_{p_1}\left(\frac{ Jf \partial_{\vec{x}} V\cdot\hat{e}_{p_1}}{h_1} \right)  +
\partial_{p_2}\left( \frac{Jf \partial_{\vec{x}} V\cdot\hat{e}_{p_2}}{h_2} \right)  +
\partial_{p_3}\left( \frac{ Jf \partial_{\vec{x}} V\cdot\hat{e}_{p_3}}{h_3} \right)  
\right]
= C(f) ,
\end{equation}
\begin{equation}
C(f) = J Q(f)
= J \int_{\Omega_{\vec{p}}} S(\vec{p}\,' \rightarrow \vec{p} )     
J' f' \, dp_1'\,dp_2'\,dp_3'
- \, Jf \int_{\Omega_{\vec{p}}} S(\vec{p} \rightarrow \vec{p}\,' ) 
J' \, dp_1'\,dp_2'\,dp_3'
\end{equation}

with 
$h_j = \left| \frac{\partial \vec{p}}{ \partial p_j} \right|,
\, j=1,2,3$, $h_1 h_2 h_3 = J = \frac{\partial \vec{p}}{\partial(p_1,p_2,p_3)}$ the jacobian of the transformation,
$J' = \frac{\partial \vec{p}\,'}{\partial(p_1',p_2',p_3')}$,
and $\hat{e}_j$ the unitary vectors associated to each 
curvilinear coordinate $p_j$ at the point $(p_1,p_2,p_3)$.

We notice that we have expressed the Boltzmann Eq. in divergence form with respect to the momentum curvilinear coordinates.
We can write it even more compactly in the form
\begin{equation}
\partial_t (Jf) +  \partial_{\vec{x}}   \cdot \left( J f  
\vec{v}(\vec{p}\,) \right)
+  \sum_{j=1}^3 
\partial_{p_j}\left( Jf
\frac{  q\partial_{\vec{x}} V(\vec{x},t) \cdot\hat{e}_{p_j}}{h_j} \right)  
= C(f) ,
\end{equation}

If $J \geq 0$, we can interpret $Jf(\vec{x},p_1,p_2,p_3,t)$ as a probability density function in the phase space $(\vec{x},p_1,p_2,p_3)$

This Boltzmann Eq. is a more general form for orthogonal curvilinear  coordinates, from which our previous spherical coordinate systems from Chapter 3 and Chapter 4 can be derived. For the one in Chapter 3, the orthogonal curvilinear system is $(r,\mu, \varphi)$, with $r \propto k^2.$ The one in Chapter 4 is $(w,\mu, \varphi)$, with
$w \propto \varepsilon$. assuming a Kane band energy $\varepsilon$.



\section{1Dx-2Dp Diode Symmetric Problem}

As we have mentioned, for the case of a 1D silicon diode, the main collision mechanisms are electron-phonon scatterings
\begin{equation}
S(\vec{p}\,' \rightarrow \vec{p}) = \sum_{j=-1}^{+1} c_j \delta(\varepsilon(\vec{p}\,') - \varepsilon(\vec{p}) + j\hbar \omega), \quad c_1 = (n_{ph} + 1) K, \, c_{-1} = n_{ph} K,
\end{equation}
with $\omega$ the phonon frequency, assumed constant, and $n_{ph}=n_{ph}(\omega)$ the phonon density. $K, \, c_0$ are constants. 

If we assume that the energy band just depends on the momentum norm, $\varepsilon(p), \quad p = |\vec{p}|$, and that the initial condition for the pdf has azimuthal symmetry, $f|_{t=0} = f_0(x,p, \mu, t), \, \partial_{\varphi} f = 0$,   $\quad \vec{p} = p (\mu, \sqrt{1-\mu^2} \cos\varphi, \sqrt{1-\mu^2} \sin \varphi) $, then the dimensionality of the problem is reduced to 3D+time, 1D in $x$, 2D in $(p,\mu)$, and the BP system for $f(x,p,\mu,t), \, V(x,t)$ is written in spherical coordinates $\vec{p}(p,\mu,\varphi)$ for the momentum as

\begin{equation}
\partial_t f + \partial_x (f \partial_p \varepsilon \mu) + \left[ \frac{\partial_p (p^2 f \mu) }{p^2} + \frac{\partial_{\mu} (f(1-\mu^2)) }{p} \right] q\partial_x V(x,t) = Q(f) \, ,
\end{equation}

\begin{equation}
- \partial_x^2 V = \frac{q}{\epsilon} \left[ N(x) - 
2\pi \int_{-1}^{+1} \int_{0}^{p_{max}} f p^2 dp d\mu, \right], \,
\quad V(0) = 0, \, V(L) = V_0 .
\end{equation}

We have assumed that the permittivity $\epsilon$ is constant. 
The Poisson BVP above can be easily solved and an analytic integral solution is easily obtained for $V(x,t)$ and $E(x,t) = -\partial_x V(x,t)$, which later can be projected in the adequate space for the numerical method. For this problem we only need to concern about the Boltzmann Equation, since given the electron density we know the solution for the potential and electric field. 

The collision operator, in this case, has the form

\begin{eqnarray}
&& 
Q(f) =
2\pi \left[ 
\sum_{j=-1}^{+1} c_j  \int_{-1}^{+1} d\mu' \left. 
f(x,p(\varepsilon'),\mu') \, p^2(\varepsilon') \frac{dp'}{d\varepsilon'} \right|_{\varepsilon' = \varepsilon(p) + j\hbar\omega } \chi(\varepsilon(p) + j\hbar\omega) 
\right.
\nonumber\\
&& - \left.
f(x,p,\mu,t)  \sum_{j=-1}^{+1} c_j \, 2 \left. p^2(\varepsilon') \frac{dp'}{d\varepsilon'} \right|_{\varepsilon' = \varepsilon(p) - j\hbar\omega } \chi(\varepsilon(p) - j\hbar\omega)
\right]
\nonumber
\end{eqnarray}

where $\chi(\varepsilon)$ is 1 if $\varepsilon \in [0,\varepsilon_{max}]$ and 0 if $\varepsilon \notin [0,\varepsilon_{max}]$, with $\varepsilon_{max} = \varepsilon(p_{max})$. 

The domain of the BP problem is $x\in [0,L], \, p \in [0, p_{max}], \, \mu \in [-1,+1]$, $t>0$.

Moreover, since $\varepsilon(p)$, 
then $\partial_{\vec{p}} \varepsilon = \frac{d\varepsilon}{dp} \hat{p} $.
We assume that $\frac{d\varepsilon}{dp} > 0$ is well behaved enough such that $p(\varepsilon)$ is a monotonic function for which
$\frac{dp}{d\varepsilon} = ( \frac{d\varepsilon}{dp} )^{-1} $ exists. 

The collision frequency is

\begin{equation}
\nu(\varepsilon(p)) = 
\sum_{j=-1}^{+1} c_j \, 4 \pi \, 
\chi(\varepsilon(p) - j\hbar\omega) \,
\left. p^2(\varepsilon') \frac{dp'}{d\varepsilon'} \right|_{\varepsilon' = \varepsilon(p) - j\hbar\omega } 
=
\sum_{j=-1}^{+1} c_j n(\varepsilon(p) - j \hbar \omega)
\end{equation}

where 

\begin{equation}
n(\varepsilon(p) - j \hbar \omega) = \int_{\Omega_{\vec{p}}} \delta(\varepsilon(\vec{p}\,') - \varepsilon(\vec{p}) + j \hbar \omega) \, d\vec{p}\, '  
\end{equation}

is the density of states with energy $\varepsilon(p) - j \hbar \omega$.


\section{DG for Boltzmann-Poisson 1Dx-2Dp Problem}

\subsection{Weak Form of the Transformed Boltzmann Eq.}

Since for $f(x,p,\mu), \, g(x,p,\mu) $ we have that

\begin{equation}
\int_{\Omega_x} \int_{\Omega_{\vec{p}}} f g \, d\vec{p} dx = 
2\pi \, \int_{\Omega_x} \int_{\Omega_{(p,\mu)}} f g \, p^2 \, dp d\mu dx
\end{equation}

we define the inner product of two functions $f$ and $g$
in the $(x,p,\mu)$ space as
\begin{equation}
(f,g)_{X \times K} = \int_X \int_K f g \, p^2 \, dp d\mu dx
\end{equation}
where $X \subset [0,L] $ and $K \subset [0,p_{max}]\times[-1,+1]$.

The Boltzmann Equation for our problem can be written in weak form as 
\begin{equation}
( \partial_t f,g)_{\Omega} + \left( \partial_x (f \partial_p \varepsilon \mu),\, g \right)_{\Omega} + \left( \left[ \frac{\partial_p (p^2 f \mu) }{p^2} + \frac{\partial_{\mu} (f(1-\mu^2)) }{p} \right] q\partial_x V(x,t) , g\right)_{\Omega} = \left( Q(f), g \right)_{\Omega} \, ,
\end{equation}

where $\Omega = X \times K $. 
More specifically, we have that ($\partial_t g = 0 $)

\begin{eqnarray}
&& \partial_t
\int_{\Omega} f\,g \, p^2 \, dp d\mu dx \, +
\int_{\Omega} \partial_x (f \partial_p \varepsilon \mu) \, g \, p^2 dp d\mu dx  
\nonumber\\
&+& 
\int_{\Omega}
{\partial_p (p^2 f \mu) } q\partial_x V(x,t) g \, dp d\mu dx + 
\int_{\Omega}
{\partial_{\mu} (f(1-\mu^2)) }  q\partial_x V(x,t) g \, p \, dp d\mu dx   \nonumber\\
& = & 
\int_{\Omega}  Q(f) \, g \, p^2 dp d\mu dx
\nonumber
\end{eqnarray}

\subsection{DG-FEM Formulation for the Transformed Boltzmann Eq. in the $(x,p,\mu)$ domain}

We will use the following mesh in the domain
\begin{equation}
\Omega_{ikm} =
X_i \times K_{k,m} = 
[x_{i^-}, x_{i^+}] \times [p_{k^-}, p_{k^+}] \times 
[\mu_{m^-}, \mu_{m^+}] 
\end{equation}
where 
\begin{equation}
x_{i^\pm} = x_{i\pm 1/2}, \quad
p_{k^\pm} = p_{k\pm 1/2}, \quad 
\mu_{m^\pm} = \mu_{m\pm 1/2} \, .
\end{equation}
We define the following notation for the internal product in our problem using the above mentioned mesh:
\begin{equation}
\int_{ikm} f g \, p^2 dp d\mu \, dx = 
(f,g)_{\Omega_{ikm}} 
\end{equation}


The semi-discrete DG Formulation for our Transformed Boltzmann Equation in curvilinear coordinates is: \\

Find $f_h \, \in \, V_h^k$ such that $\forall \, g_h \, \in V_h^k$ 
and $\, \forall \, \Omega_{ikm} $

\begin{eqnarray}
\partial_t  \int_{ikm}  f_h \, g_h \, p^2 dp d\mu dx 
&&
\nonumber\\
-\int_{ikm}  \partial_p \varepsilon (p) \, f_h \, \mu \, \partial_x g_h \, p^2 dp d\mu dx \,
&\pm &
 \int_{km} \partial_p \varepsilon \, \widehat{ f_h  \mu }|_{x_{i\pm}}  \, g_h|_{x_{i\pm}}^{\mp} \, p^2 dp d\mu  
\nonumber\\
- \int_{ikm}
{ p^2 } (-qE)(x,t) f_h \mu \, \partial_p g_h \, d\mu dx
& \pm & 
\int_{im} p^2_{k^\pm} \,
(-q  \widehat{ E f_h \mu } )|_{p_{k\pm}} g_h|_{p_{k\pm}}^{\mp} \, d\mu dx
\nonumber\\
- \int_{ikm}
{ (1-\mu^2) f_h }  (-qE)(x,t) \, \partial_{\mu} g_h \, p \, dp d\mu dx 
&\pm &
\int_{ik} (1-\mu_{m\pm}^2)
(-q \widehat{ E f_h })|_{\mu_{m\pm}}   \, g_h|_{\mu_{m\pm}}^{\mp} \, p \, dp dx
\nonumber\\
&=&
\int_{ikm} Q(f_h) g_h \, p^2 dp d\mu dx   \, .
\nonumber
\end{eqnarray}

The Numerical Flux used is the Upwind Rule. Therefore we have that

\begin{eqnarray}
\widehat{ f_h  \mu }|_{x_{i\pm}} &=& 
\left(\frac{\mu + |\mu|}{2} \right) f_h |^-_{x_{i\pm}} +
\left(\frac{\mu - |\mu|}{2} \right) f_h |^+_{x_{i\pm}}
\nonumber\\
- \widehat{ qE \mu f_h  } |_{p_{k\pm}} & =&
\left(\frac{- qE \mu  + |qE \mu |}{2}\right) f_h |^-_{p_{k\pm}} +
\left(\frac{- qE \mu  - |qE \mu |}{2} \right) f_h |^+_{p_{k\pm}}
\nonumber\\
- \widehat{ qE f_h }|_{\mu_{m\pm}} & = &
\left(\frac{-qE + |qE|}{2} \right) f_h |^-_{\mu_{m\pm}} +
\left(\frac{-qE - |qE|}{2} \right) f_h |^+_{\mu_{m\pm}}
\nonumber
\end{eqnarray}

Using the notation in the paper of Eindeve, Hauck, Xing, Mezzacappa \cite{EECHXM-JCP}, the semi-discrete DG formulation is written as: \\

Find $f_h \, \in \, V_h^k$ such that $\forall \, g_h \, \in V_h^k$ 
and $\, \forall \, \Omega_{ikm} $

\begin{eqnarray}
\partial_t  \int_{\Omega_{ikm}}  f_h \, g_h \, p^2 dp d\mu dx 
&&
\nonumber\\
-\int_{\Omega_{ikm}} H^{(x)}  \, f_h \, \partial_x g_h \, p^2 dp d\mu dx \,
&\pm &
 \int_{\tilde{\Omega}^{(x)}_{km}}  \,  \widehat{H^{(x)}f_h}|_{x_{i\pm}} \, g_h|_{x_{i\pm}}^{\mp} \, p^2 dp d\mu  
\nonumber\\
- \int_{\Omega_{ikm}} { p^2 } \,
H^{(p)}  f_h \, \partial_p g_h \, d\mu dx
& \pm & 
p^2_{k^\pm}
\int_{\tilde{\Omega}^{(p)}_{im}}  \,
\widehat{H^{(p)}f_h}|_{p_{k\pm}} \, g_h|_{p_{k\pm}}^{\mp} \, d\mu dx
\nonumber\\
- \int_{\Omega_{ikm}}
{ (1-\mu^2) H^{(\mu)} f_h } \, \partial_{\mu} g_h \, p \, dp d\mu dx 
&\pm &
(1-\mu_{m\pm}^2)
\int_{\tilde{\Omega}^{(\mu)}_{ik}} 
\widehat{H^{(\mu)}f_h}|_{\mu_{m\pm}} \, g_h|_{\mu_{m\pm}}^{\mp} \, p \, dp dx
\nonumber\\
&=&
\int_{\Omega_{ikm}} Q(f_h) g_h \, p^2 dp d\mu dx   \, ,
\nonumber
\end{eqnarray}

where we have defined the terms ($\partial_p \varepsilon(p)>0$):

\begin{eqnarray}
H^{(x)}(p,\mu) = \mu \, \partial_p \varepsilon (p) 
\, , 
&
 \widehat{H^{(x)}f}|_{x_{i\pm}}
= \partial_p \varepsilon
\widehat{ f_h  \mu }|_{x_{i\pm}} \, ,
&
\tilde{\Omega}^{(x)}_{km} = [r_{k-},r_{k+}]\times[\mu_{m-},\mu_{m+}] 
= \partial_x \Omega_{km}
\, ,
\nonumber\\
H^{(p)}(t,x,\mu) = -q E(x,t) \mu
\, ,
&
\widehat{H^{(p)}f}|_{p_{k\pm}}
=
 -q \widehat{ E f_h \mu } |_{p_{k\pm}} 
\, ,
&
\tilde{\Omega}^{(p)}_{im} = [x_{i-},x_{i+}]\times[\mu_{m-},\mu_{m+}] 
= \partial_p \Omega_{im}
\, ,
\nonumber\\
H^{(\mu)}(x,t) 
=
-qE(x,t)
\, ,
&
\widehat{H^{(\mu)}f}|_{\mu_{m\pm}}
=
-q  \widehat{ E f_h }|_{\mu_{m\pm}} 
\, ,
&
\tilde{\Omega}^{(\mu)}_{ik} = [x_{i-},x_{i+}]\times[r_{k-},r_{k+}] 
= \partial_{\mu}\Omega_{ik}
\, ,
\nonumber
\end{eqnarray}

The weak form of the collisional operator in the DG scheme is, specifically

\begin{eqnarray}
&&
\int_{\Omega_{ikm}} Q(f_h) \, g_h \, p^2 \, dp d\mu dx   \, 
 = 
\int_{\Omega_{ikm}} \left[ G(f_h) - \nu(\varepsilon(p)) f_h \right] \, g_h \, p^2 \, dp d\mu dx   \, =
\nonumber\\
&&
2\pi 
\int_{\Omega_{ikm}} 
\left( 
\sum_{j=-1}^{+1} c_j \, \chi(\varepsilon(p) + j\hbar\omega)
\int_{-1}^{+1} d\mu' \left. \left[
f_h(x,p(\varepsilon'),\mu') \, p^2(\varepsilon') \frac{dp'}{d\varepsilon'} \right] \right|_{\varepsilon' = \varepsilon(p) + j\hbar\omega }  
\right) 
g_h \, p^2 \, dp d\mu dx  
\nonumber\\
&&
-
4\pi 
\int_{\Omega_{ikm}} 
f_h(x,p,\mu,t)  
\left(
\sum_{j=-1}^{+1} c_j \, \chi(\varepsilon(p) - j\hbar\omega)  \left. \left[ p^2(\varepsilon') \frac{dp'}{d\varepsilon'} \right] \right|_{\varepsilon' = \varepsilon(p) - j\hbar\omega } 
\right)
g_h \, p^2 \, dp d\mu dx  
\end{eqnarray}

The cell average of $f_h$ in $\Omega_{ikm}$ is

\begin{equation}
\bar{f}_{ikm} = \frac{\int_{\Omega_{ikm}} f_h \,  p^2 \, dp d\mu dx }{\int_{\Omega_{ikm}} p^2 \, dp d\mu dx} = 
\frac{\int_{\Omega_{ikm}} f_h \, dV  }{ V_{ikm} }, \,
\end{equation}

where, for our particular curvilinear coordinates, spherical:

\begin{equation}
V_{ikm} = \int_{\Omega_{ikm}} dV \, , \, 
dV = \tau \prod_{d=1}^3 z_d ,
\quad
(z_1,z_2,z_3) = \mathbf{z} =
(x,p,\mu), 
\quad
\tau = \sqrt{\gamma\lambda}, \,
\gamma = 1, \, \lambda = p^2
\end{equation}

The time evolution of the cell average in the DG scheme is given by

\begin{eqnarray}
&&
\partial_t \bar{f}_{ikm} = 
\nonumber\\
& - & \frac{1}{V_{ikm}} \left[
 \int_{\partial_x \Omega_{km}}  \,  \widehat{H^{(x)}f_h}|_{x_{i+}} 
 \, p^2 dp d\mu  
-
 \int_{\partial_x \Omega_{km}}  \,  \widehat{H^{(x)}f_h}|_{x_{i-}}  
 \, p^2 dp d\mu  
\right. 
\nonumber\\
& + & 
p^2_{k^+}
\int_{\partial_p \Omega_{im}}  \,
\widehat{H^{(p)}f_h}|_{p_{k+}} 
\, d\mu dx
-  
p^2_{k^-}
\int_{\partial_p \Omega_{im}}  \,
\widehat{H^{(p)}f_h}|_{p_{k-}} 
\, d\mu dx
\nonumber\\
&+ &
\left.
(1-\mu_{m+}^2)
\int_{\partial_{\mu}\Omega_{ik}} 
\widehat{H^{(\mu)}f_h}|_{\mu_{m+}}
\, p \, dp dx
-
(1-\mu_{m-}^2)
\int_{\partial_{\mu}\Omega_{ik}} 
\widehat{H^{(\mu)}f_h}|_{\mu_{m-}} 
\, p \, dp dx
\right]
\nonumber\\
& + &
\left[
{2\pi}  
\int_{\Omega_{ikm}} 
\left( 
\sum_{j=-1}^{+1} c_j \, \chi(\varepsilon(p) + j\hbar\omega)
\int_{-1}^{+1} d\mu' \left. \left[
f_h(x,p(\varepsilon'),\mu') \, p^2(\varepsilon') \frac{dp'}{d\varepsilon'} \right] \right|_{\varepsilon' = \varepsilon(p) + j\hbar\omega }  
\right) 
\, p^2 \, dp d\mu dx  
\right.
\nonumber\\
&&
\left.
-
{4\pi} 
\int_{\Omega_{ikm}} 
f_h(x,p,\mu,t)  
\left(
\sum_{j=-1}^{+1} c_j \, \chi(\varepsilon(p) - j\hbar\omega)  \left. \left[ p^2(\varepsilon') \frac{dp'}{d\varepsilon'} \right] \right|_{\varepsilon' = \varepsilon(p) - j\hbar\omega } 
\right)
\, p^2 \, dp d\mu dx  
\right]\frac{1}{V_{ikm}}
\nonumber
\end{eqnarray}

Regarding the time discretization, we will apply a TVD RK-DG scheme. These schemes are convex combinations of Euler methods. 
We consider therefore the time evolution of the cell average in the DG scheme using Forward Euler: $\partial_t \bar{f}_{ikm} \approx (\bar{f}_{ikm}^{n+1} - \bar{f}_{ikm}^n)/\Delta t^n $ 


\begin{eqnarray}
&&
\bar{f}_{ikm}^{n+1} = \bar{f}_{ikm}^n
\nonumber\\
& - & \frac{\Delta t^n}{V_{ikm}} \left[
 \int_{\partial_x \Omega_{km}}  \,  \widehat{H^{(x)}f_h}|_{x_{i+}} 
 \, p^2 dp d\mu  
-
 \int_{\partial_x \Omega_{km}}  \,  \widehat{H^{(x)}f_h}|_{x_{i-}}  
 \, p^2 dp d\mu  
\right. 
\nonumber\\
& + & 
p^2_{k^+}
\int_{\partial_p \Omega_{im}}  \,
\widehat{H^{(p)}f_h}|_{p_{k+}} 
\, d\mu dx
-  
p^2_{k^-}
\int_{\partial_p \Omega_{im}}  \,
\widehat{H^{(p)}f_h}|_{p_{k-}} 
\, d\mu dx
\nonumber\\
&+ &
\left.
(1-\mu_{m+}^2)
\int_{\partial_{\mu}\Omega_{ik}} 
\widehat{H^{(\mu)}f_h}|_{\mu_{m+}}
\, p \, dp dx
-
(1-\mu_{m-}^2)
\int_{\partial_{\mu}\Omega_{ik}} 
\widehat{H^{(\mu)}f_h}|_{\mu_{m-}} 
\, p \, dp dx
\right]
\nonumber\\
& + &
\left[
{2\pi}  
\int_{\Omega_{ikm}} 
\left( 
\sum_{j=-1}^{+1} c_j \, \chi(\varepsilon(p) + j\hbar\omega)
\int_{-1}^{+1} d\mu' \left. \left[
f_h(x,p(\varepsilon'),\mu') \, p^2(\varepsilon') \frac{dp'}{d\varepsilon'} \right] \right|_{\varepsilon' = \varepsilon(p) + j\hbar\omega }  
\right) 
\, p^2 \, dp d\mu dx  
\right.
\nonumber\\
&&
\left.
-
{4\pi} 
\int_{\Omega_{ikm}} 
f_h(x,p,\mu,t)  
\left(
\sum_{j=-1}^{+1} c_j \, \chi(\varepsilon(p) - j\hbar\omega)  \left. \left[ p^2(\varepsilon') \frac{dp'}{d\varepsilon'} \right] \right|_{\varepsilon' = \varepsilon(p) - j\hbar\omega } 
\right)
\, p^2 \, dp d\mu dx  
\right]\frac{\Delta t^n}{V_{ikm}}
\nonumber
\end{eqnarray}

or, more briefly

\begin{equation}
\bar{f}_{ikm}^{n+1} = \bar{f}_{ikm}^n + \Gamma_T + \Gamma_C
\end{equation}

where the transport and collision terms for the cell average time evolution are defined as

\begin{eqnarray}
\Gamma_T
& = & - \frac{\Delta t^n}{V_{ikm}} \left[
 \int_{\partial_x \Omega_{km}}  \,  \widehat{H^{(x)}f_h}|_{x_{i+}} 
 \, p^2 dp d\mu  
-
 \int_{\partial_x \Omega_{km}}  \,  \widehat{H^{(x)}f_h}|_{x_{i-}}  
 \, p^2 dp d\mu  
\right. 
\nonumber\\
& + & 
p^2_{k^+}
\int_{\partial_p \Omega_{im}}  \,
\widehat{H^{(p)}f_h}|_{p_{k+}} 
\, d\mu dx
-  
p^2_{k^-}
\int_{\partial_p \Omega_{im}}  \,
\widehat{H^{(p)}f_h}|_{p_{k-}} 
\, d\mu dx
\nonumber\\
&+ &
\left.
(1-\mu_{m+}^2)
\int_{\partial_{\mu}\Omega_{ik}} 
\widehat{H^{(\mu)}f_h}|_{\mu_{m+}}
\, p \, dp dx
-
(1-\mu_{m-}^2)
\int_{\partial_{\mu}\Omega_{ik}} 
\widehat{H^{(\mu)}f_h}|_{\mu_{m-}} 
\, p \, dp dx
\right]
\nonumber
\end{eqnarray}

\begin{eqnarray}
\Gamma_C
& = &
\left[
{2\pi}  
\int_{\Omega_{ikm}} 
\left( 
\sum_{j=-1}^{+1} c_j \, \chi(\varepsilon(p) + j\hbar\omega)
\int_{-1}^{+1} d\mu' \left. \left[
f_h(x,p(\varepsilon'),\mu') \, p^2(\varepsilon') \frac{dp'}{d\varepsilon'} \right] \right|_{\varepsilon' = \varepsilon(p) + j\hbar\omega }  
\right) 
\, p^2 \, dp d\mu dx  
\right.
\nonumber\\
&&
\left.
-
{4\pi} 
\int_{\Omega_{ikm}} 
f_h(x,p,\mu,t)  
\left(
\sum_{j=-1}^{+1} c_j \, \chi(\varepsilon(p) - j\hbar\omega)  \left. \left[ p^2(\varepsilon') \frac{dp'}{d\varepsilon'} \right] \right|_{\varepsilon' = \varepsilon(p) - j\hbar\omega } 
\right)
\, p^2 \, dp d\mu dx  
\right]\frac{\Delta t^n}{V_{ikm}}
\nonumber
\end{eqnarray}

\subsection{Positivity Preservation in DG Scheme for BP}

We use the strategy of Zhang \& Shu in \cite{ZhangShu1}, \cite{ZhangShu2}, for conservation laws, Eindeve, Hauck, Xing, Mezzacappa \cite{EECHXM-JCP} for conservative phase space advection in curvilinear
coordinates, and Cheng, Gamba, Proft for Vlasov-Boltzmann with a linear non-degenerate collisional forms  \cite{CGP} to preserve the positivity of our probability density function in our DG scheme treating the collision term as a source, this being possible as our collisional form is mass preserving. 
We will use a convex combination parameter $\alpha \in [0,1]$
\begin{equation}
\bar{f}_{ikm}^{n+1} = 
\alpha 
\underbrace{
\left( \bar{f}_{ikm}^n + \frac{\Gamma_T}{\alpha} \right) 
}_{I}
+ 
(1 - \alpha)
\underbrace{
\left( \bar{f}_{ikm}^n + \frac{\Gamma_C}{1-\alpha}\right)  
}_{II}
\end{equation}
and we will find conditions such that $I$ and $II$ are positive,
to guarantee the positivity of the cell average of our numerical probability density function for the next time step. 
The positivity of the numerical solution to the pdf in the whole domain can be guaranteed just by applying the limiters in \cite{ZhangShu1}, \cite{ZhangShu2} that preserve the cell average but modify the slope of the piecewise linear solutions in order to make the function non - negative.




Regarding $I$, the conditions for its positivity are derived below.

\begin{eqnarray}
I & = &
 \,  \bar{f}_{ikm}^n + \frac{\Gamma_T}{\alpha}   
\, = \,
\frac{\int_{\Omega_{ikm}} f_h \,  p^2 \, dp d\mu dx }{ V_{ikm} }
\nonumber\\
&&
-  \frac{\Delta t^n}{ \alpha V_{ikm}} \left[
 \int_{\partial_x \Omega_{km}}  \,  \widehat{H^{(x)}f_h}|_{x_{i+}} 
 \, p^2 dp d\mu  
-
 \int_{\partial_x \Omega_{km}}  \,  \widehat{H^{(x)}f_h}|_{x_{i-}}  
 \, p^2 dp d\mu  
\right. 
\nonumber\\
& & +  
p^2_{k^+}
\int_{\partial_p \Omega_{im}}  \,
\widehat{H^{(p)}f_h}|_{p_{k+}} 
\, d\mu dx
-  
p^2_{k^-}
\int_{\partial_p \Omega_{im}}  \,
\widehat{H^{(p)}f_h}|_{p_{k-}} 
\, d\mu dx
\nonumber\\
&& +
\left.
(1-\mu_{m+}^2)
\int_{\partial_{\mu}\Omega_{ik}} 
\widehat{H^{(\mu)}f_h}|_{\mu_{m+}}
\, p \, dp dx
-
(1-\mu_{m-}^2)
\int_{\partial_{\mu}\Omega_{ik}} 
\widehat{H^{(\mu)}f_h}|_{\mu_{m-}} 
\, p \, dp dx
\right]
\nonumber
\end{eqnarray}

We will split the cell average using 3 convex parameters
$s_l \geq 0 , \, l=1,2,3 \,$ s.t. $s_1 + s_2 + s_3 = 1$.
We have then

\begin{eqnarray}
I &=& 
 \frac{ 1 }{ V_{ikm} }
\left[
(s_1 + s_2 + s_3) \int_{\Omega_{ikm}} f_h \,  p^2 \, dp d\mu dx
\right.
\nonumber\\
&&
-  \frac{\Delta t^n}{ \alpha } \left(
 \int_{\partial_x \Omega_{km}}  \,  \widehat{H^{(x)}f_h}|_{x_{i+}} 
 \, p^2 dp d\mu  
-
 \int_{\partial_x \Omega_{km}}  \,  \widehat{H^{(x)}f_h}|_{x_{i-}}  
 \, p^2 dp d\mu  
\right. 
\nonumber\\
& & +  
p^2_{k^+}
\int_{\partial_p \Omega_{im}}  \,
\widehat{H^{(p)}f_h}|_{p_{k+}} 
\, d\mu dx
-  
p^2_{k^-}
\int_{\partial_p \Omega_{im}}  \,
\widehat{H^{(p)}f_h}|_{p_{k-}} 
\, d\mu dx
\nonumber\\
&& +
\left.
\left.
(1-\mu_{m+}^2)
\int_{\partial_{\mu}\Omega_{ik}} 
\widehat{H^{(\mu)}f_h}|_{\mu_{m+}}
\, p \, dp dx
-
(1-\mu_{m-}^2)
\int_{\partial_{\mu}\Omega_{ik}} 
\widehat{H^{(\mu)}f_h}|_{\mu_{m-}} 
\, p \, dp dx
\right)
\right]
\nonumber
\end{eqnarray}
\begin{eqnarray}
& = &
 \frac{ 1 }{ V_{ikm} }
\left[
s_1 \int_{x_{i-}}^{x_{i+}} \int_{\partial_x \Omega_{km}} f_h \,  p^2 \, dp d\mu dx
-  \frac{\Delta t^n}{ \alpha } \left(
 \int_{\partial_x \Omega_{km}}  \,  \widehat{H^{(x)}f_h}|_{x_{i+}} 
 \, p^2 dp d\mu  
-
 \int_{\partial_x \Omega_{km}}  \,  \widehat{H^{(x)}f_h}|_{x_{i-}}  
 \, p^2 dp d\mu  
\right) 
\right.
\nonumber\\
& & +  s_2 \int_{p_{k-}}^{p_{k+}} \int_{\partial_p \Omega_{im}} f_h \,  p^2 \, dp d\mu dx
-  \frac{\Delta t^n}{ \alpha } \left(
p^2_{k^+}
\int_{\partial_p \Omega_{im}}  \,
\widehat{H^{(p)}f_h}|_{p_{k+}} 
\, d\mu dx
-  
p^2_{k^-}
\int_{\partial_p \Omega_{im}}  \,
\widehat{H^{(p)}f_h}|_{p_{k-}} 
\, d\mu dx
\right)
\nonumber\\
&& +
s_3 \int_{\mu_{m-}}^{\mu_{m+}} \int_{\partial_{\mu}\Omega_{ik}} 
 f_h \,  p^2 \, dp d\mu dx
\nonumber\\
&&
\left.
-  \frac{\Delta t^n}{ \alpha } \left(
(1-\mu_{m+}^2)
\int_{\partial_{\mu}\Omega_{ik}} 
\widehat{H^{(\mu)}f_h}|_{\mu_{m+}}
\, p \, dp dx
-
(1-\mu_{m-}^2)
\int_{\partial_{\mu}\Omega_{ik}} 
\widehat{H^{(\mu)}f_h}|_{\mu_{m-}} 
\, p \, dp dx
\right)
\right]
\nonumber\\
& = &
 \frac{ 1 }{ V_{ikm} }
\left[
 \int_{\partial_x \Omega_{km}} 
\left\lbrace 
s_1 \int_{x_{i-}}^{x_{i+}} f_h \,  p^2 \,  dx
-  \frac{\Delta t^n}{ \alpha } \left(
\widehat{H^{(x)}f_h}|_{x_{i+}} 
 \, p^2   
-
\widehat{H^{(x)}f_h}|_{x_{i-}}  
 \, p^2   
\right)
\right\rbrace 
dp \, d\mu 
\right.
\nonumber\\
& & +   \int_{\partial_p \Omega_{im}} 
\left\lbrace
s_2 \int_{p_{k-}}^{p_{k+}}
f_h \,  p^2 \, dp 
-  \frac{\Delta t^n}{ \alpha } \left(
p^2_{k^+}
\widehat{H^{(p)}f_h}|_{p_{k+}} 
-  
p^2_{k^-}
\widehat{H^{(p)}f_h}|_{p_{k-}} 
\right)
\right\rbrace
d\mu \, dx
\nonumber\\
&& +
\left.
 \int_{\partial_{\mu}\Omega_{ik}} 
\left\lbrace 
s_3 \int_{\mu_{m-}}^{\mu_{m+}}
 f_h \,  p^2 \, d\mu 
-  \frac{\Delta t^n}{ \alpha } \left[
(1-\mu_{m+}^2)
\widehat{H^{(\mu)}f_h}|_{\mu_{m+}}
\, p 
\, - \,
(1-\mu_{m-}^2)
\widehat{H^{(\mu)}f_h}|_{\mu_{m-}} 
\, p 
\right]
\right\rbrace \, dp dx
\right]
\nonumber
\end{eqnarray}

All the functions to be integrated are polynomials inside a given interval, rectangle or element. Therefore, we can integrate them exactly using a quadrature rule of enough degree, which could be either the usual Gaussian quadrature or the Gauss-Lobatto, which involves the end-points of the interval. 
We use Gauss-Lobatto quadratures for the integrals of $f_h \, p^2$ over intervals, so that the values at the endpoints can balance the flux terms of boundary integrals, obtaining then CFL conditions.

\begin{eqnarray}
I & = &
 \frac{ 1 }{ V_{ikm} }
\left[
 \int_{\partial_x \Omega_{km}} 
\left\lbrace 
s_1 
\sum_{\hat{q}=1}^N \hat{w}_{\hat{q}}
f_h|_{x_{\hat{q}}} \,  p^2 \,  \Delta x_i
-  \frac{\Delta t^n}{ \alpha } \left(
\widehat{H^{(x)}f_h}|_{x_{i+}} 
 \, p^2   
-
\widehat{H^{(x)}f_h}|_{x_{i-}}  
 \, p^2   
\right)
\right\rbrace 
dp \, d\mu 
\right.
\nonumber\\
& & +   \int_{\partial_p \Omega_{im}} 
\left\lbrace
s_2 
\sum_{\hat{r}=1}^N \hat{w}_{\hat{r}}
f_h|_{p_{\hat{r}}} \,  p^2_{\hat{r}} \, \Delta p_k 
-  \frac{\Delta t^n}{ \alpha } \left(
p^2_{k^+}
\widehat{H^{(p)}f_h}|_{p_{k+}} 
-  
p^2_{k^-}
\widehat{H^{(p)}f_h}|_{p_{k-}} 
\right)
\right\rbrace
d\mu \, dx
\nonumber\\
&& +
\left.
 \int_{\partial_{\mu}\Omega_{ik}} 
\left\lbrace 
s_3 
\sum_{\hat{s}=1}^N \hat{w}_{\hat{s}}
 f_h|_{\mu_{\hat{s}}} \,  p^2 \, \Delta \mu_m 
-  \frac{\Delta t^n}{ \alpha } \left[
(1-\mu_{m+}^2)
\widehat{H^{(\mu)}f_h}|_{\mu_{m+}}
\, p 
\, - \,
(1-\mu_{m-}^2)
\widehat{H^{(\mu)}f_h}|_{\mu_{m-}} 
\, p 
\right]
\right\rbrace \, dp dx
\right]
\nonumber\\
 & = &
\left[
 \int_{\partial_x \Omega_{km}} 
\left\lbrace 
s_1 \Delta x_i  
\left(
\hat{w}_1 f_h|_{x_{i-}}^+  +
\hat{w}_N f_h|_{x_{i+}}^-  + 
  \sum_{\hat{q}=2}^{N-1} \hat{w}_{\hat{q}}
f_h|_{x_{\hat{q}}}  
\right)
-  \frac{\Delta t^n}{ \alpha } \left(
\widehat{H^{(x)}f_h}|_{x_{i+}} 
-
\widehat{H^{(x)}f_h}|_{x_{i-}}  
\right)
\right\rbrace 
p^2 
dp d\mu 
\right.
\nonumber\\
& & +   \int_{\partial_p \Omega_{im}} 
\left\lbrace
s_2 
\left(
\hat{w}_{1} f_h|_{p_{k-}}^+ \,  p^2_{k-} +
\hat{w}_{N}  f_h|_{p_{k+}}^- \,  p^2_{k+} +
\sum_{\hat{r}=2}^{N-1} \hat{w}_{\hat{r}}
f_h|_{p_{\hat{r}}} \,  p^2_{\hat{r}} \, 
\right)
\Delta p_k 
\right.
\nonumber\\
&&
\left.
\quad \quad
-  \frac{\Delta t^n}{ \alpha } \left(
p^2_{k^+}
\widehat{H^{(p)}f_h}|_{p_{k+}} 
-  
p^2_{k^-}
\widehat{H^{(p)}f_h}|_{p_{k-}} 
\right)
\right\rbrace
d\mu \, dx
\nonumber\\
&& +
 \int_{\partial_{\mu}\Omega_{ik}} 
\left\lbrace 
s_3 
\left(
\hat{w}_{1} f_h|_{\mu_{m-}}^+
+
\hat{w}_{N} f_h|_{\mu_{m+}}^-
+
\sum_{\hat{s}=2}^{N-1} 
\hat{w}_{\hat{s}} f_h|_{\mu_{\hat{s}}}
\right) 
 p^2 \, \Delta \mu_m 
\right.
\nonumber\\
&&
\left.
\left. 
\quad \quad
-  \frac{\Delta t^n}{ \alpha } \left[
(1-\mu_{m+}^2)
\widehat{H^{(\mu)}f_h}|_{\mu_{m+}}
\, - \,
(1-\mu_{m-}^2)
\widehat{H^{(\mu)}f_h}|_{\mu_{m-}} 
\right] p
\right\rbrace \, dp dx
\right] \frac{ 1 }{ V_{ikm} }
\nonumber\\
 & = &
\left[
 \int_{\partial_x \Omega_{km}} 
s_1 \Delta x_i
\left\lbrace 
  \sum_{\hat{q}=2}^{N-1} \hat{w}_{\hat{q}}
f_h|_{x_{\hat{q}}}  
+
\left(
\hat{w}_1 f_h|_{x_{i-}}^+  +
\hat{w}_N f_h|_{x_{i+}}^-  
\right)
-  \frac{\Delta t^n}{ \alpha s_1 \Delta x_i} \left(
\widehat{H^{(x)}f_h}|_{x_{i+}} 
-
\widehat{H^{(x)}f_h}|_{x_{i-}}  
\right)
\right\rbrace 
p^2 
dp d\mu 
\right.
\nonumber\\
& & +   \int_{\partial_p \Omega_{im}} 
s_2 \Delta p_k  
\left\lbrace
\left(
\hat{w}_{1} f_h|_{p_{k-}}^+ \,  p^2_{k-}  +
\hat{w}_{N}  f_h|_{p_{k+}}^- \,  p^2_{k+} 
\right)
+
\sum_{\hat{r}=2}^{N-1} \hat{w}_{\hat{r}}
f_h|_{p_{\hat{r}}} \,  p^2_{\hat{r}} \, 
\right.
\nonumber\\
&&
\left.
\quad \quad
-  \frac{\Delta t^n}{ \alpha s_2 \Delta p_k } \left(
p^2_{k^+}
\widehat{H^{(p)}f_h}|_{p_{k+}} 
-  
p^2_{k^-}
\widehat{H^{(p)}f_h}|_{p_{k-}} 
\right)
\right\rbrace
d\mu \, dx
\nonumber\\
&& +
 \int_{\partial_{\mu}\Omega_{ik}} 
s_3 \,  p^2 \, \Delta \mu_m 
\left\lbrace 
\left(
\hat{w}_{1} f_h|_{\mu_{m-}}^+
+
\hat{w}_{N} f_h|_{\mu_{m+}}^-
\right)
+
\sum_{\hat{s}=2}^{N-1} 
\hat{w}_{\hat{s}} f_h|_{\mu_{\hat{s}}}
\right.
\nonumber\\
&&
\left.
\left. 
\quad \quad
-  \frac{\Delta t^n}{ \alpha s_3 p \, \Delta \mu_m } \left[
(1-\mu_{m+}^2)
\widehat{H^{(\mu)}f_h}|_{\mu_{m+}}
\, - \,
(1-\mu_{m-}^2)
\widehat{H^{(\mu)}f_h}|_{\mu_{m-}} 
\right] 
\right\rbrace  dp dx
\right] \frac{ 1 }{ V_{ikm} }
\nonumber
\end{eqnarray}


We reorganize the terms involving the endpoints, which are in parenthesis. So

\begin{eqnarray}
I & = &
 \frac{ 1 }{ V_{ikm} }
\left[
 \int_{\partial_x \Omega_{km}} 
s_1 \Delta x_i
\left\lbrace 
\left(
\hat{w}_1 f_h|_{x_{i-}}^+  
+  \frac{\Delta t^n}{ \alpha s_1 \Delta x_i} 
\widehat{H^{(x)}f_h}|_{x_{i-}}  
\right)
+
\left(
\hat{w}_N f_h|_{x_{i+}}^-  
-  \frac{\Delta t^n}{ \alpha s_1 \Delta x_i} 
\widehat{H^{(x)}f_h}|_{x_{i+}} 
\right)
\right.
\right.
\nonumber\\
&&
+
\left. \,
  \sum_{\hat{q}=2}^{N-1} \hat{w}_{\hat{q}}
f_h|_{x_{\hat{q}}}  
\right\rbrace 
p^2 
dp d\mu 
 +   \int_{\partial_p \Omega_{im}} 
s_2 \Delta p_k  
\left\lbrace
\sum_{\hat{r}=2}^{N-1} \hat{w}_{\hat{r}}
f_h|_{p_{\hat{r}}} \,  p^2_{\hat{r}} \, 
+
\right.
\nonumber\\
&&
\left.
\quad \quad
+
p^2_{k^-}
\left(
\hat{w}_{1} f_h|_{p_{k-}}^+ \, 
+  \frac{\Delta t^n}{ \alpha s_2 \Delta p_k } 
\widehat{H^{(p)}f_h}|_{p_{k-}} 
\right)
+
p^2_{k+} 
\left(
\hat{w}_{N}  f_h|_{p_{k+}}^- \,
-  \frac{\Delta t^n}{ \alpha s_2 \Delta p_k } 
\widehat{H^{(p)}f_h}|_{p_{k+}} 
\right)
\right\rbrace
d\mu \, dx
\nonumber\\
&& +
 \int_{\partial_{\mu}\Omega_{ik}} 
dx \, dp \, p^2 \,
s_3 \,  \Delta \mu_m 
\left\lbrace 
\sum_{\hat{s}=2}^{N-1} 
\hat{w}_{\hat{s}} f_h|_{\mu_{\hat{s}}} \, + \,
\right.
\nonumber\\
&&
\left.
\left. 
\quad 
+
\left(
\hat{w}_{1} f_h|_{\mu_{m-}}^+
+  \frac{\Delta t^n (1-\mu_{m-}^2)
}{ \alpha s_3 p \, \Delta \mu_m }
\widehat{H^{(\mu)}f_h}|_{\mu_{m-}} 
\right)
+
\left(
\hat{w}_{N} f_h|_{\mu_{m+}}^-
-  \frac{\Delta t^n (1-\mu_{m+}^2) }{ \alpha s_3 p \, \Delta \mu_m } 
\widehat{H^{(\mu)}f_h}|_{\mu_{m+}}
\right)
\right\rbrace  
\right] 
\nonumber
\end{eqnarray}

To guarantee the positivity of $I$, assuming that the terms $f_h|_{x_{\hat{q}}}, \, f_h|_{p_{\hat{r}}}, \, f_h|_{\mu_{\hat{s}}} $
are positive at time $t^n$, we only need that the terms in parenthesis related to interval endpoints are positive. 
Since $\hat{w}_1 = \hat{w}_N$ for Gauss-Lobatto Quadrature, 
we want the non-negativity of the terms 

\begin{eqnarray}
0 & \leq & 
\left(
\hat{w}_N f_h|_{x_{i\pm}}^{\mp}  
\mp  \frac{\Delta t^n}{ \alpha s_1 \Delta x_i} 
\widehat{H^{(x)}f_h}|_{x_{i\pm}} 
\right)
\nonumber\\
0 & \leq & 
\left(
\hat{w}_{N}  f_h|_{p_{k\pm}}^{\mp} \,
\mp \frac{\Delta t^n}{ \alpha s_2 \Delta p_k } 
\widehat{H^{(p)}f_h}|_{p_{k\pm}} 
\right)
\\
0 & \leq &
\left(
\hat{w}_{N} f_h|_{\mu_{m\pm}}^{\mp}
\mp  \frac{\Delta t^n (1-\mu_{m\pm}^2) }{ \alpha s_3 p \, \Delta \mu_m } 
\widehat{H^{(\mu)}f_h}|_{\mu_{m\pm}}
\right)
\nonumber
\end{eqnarray}

We remember that we have used the following notation for the numerical flux terms, given by the upwind rule

\begin{eqnarray}
 \widehat{H^{(x)}f}|_{x_{i\pm}}
&=& 
\partial_p \varepsilon
\widehat{ f_h  \mu }|_{x_{i\pm}} 
= \partial_p \varepsilon
\left[
\left(\frac{\mu + |\mu|}{2} \right) f_h |^-_{x_{i\pm}} +
\left(\frac{\mu - |\mu|}{2} \right) f_h |^+_{x_{i\pm}}
\right]
\nonumber\\
\widehat{H^{(p)}f}|_{p_{k\pm}}
&=&
- q \widehat{ E f_h \mu } |_{p_{k\pm}} 
=
q
\left[
\left(\frac{ -E(x,t) \mu  + |E(x,t) \mu |}{2}\right) f_h |^-_{p_{k\pm}} +
\left(\frac{-E(x,t) \mu  - |E(x,t) \mu |}{2} \right) f_h |^+_{p_{k\pm}}
\right]
\nonumber\\
\widehat{H^{(\mu)}f}|_{\mu_{m\pm}}
&=&
- q  \widehat{ E f_h }|_{\mu_{m\pm}} =
q
\left[
\left(\frac{-E(x,t) + |E(x,t)|}{2} \right) f_h |^-_{\mu_{m\pm}} +
\left(\frac{-E(x,t) - |E(x,t)|}{2} \right) f_h |^+_{\mu_{m\pm}}
\right]
\nonumber
\end{eqnarray}

We have assumed that the positivity of the pdf evaluated at Gauss-Lobatto points, which include endpoints, so we know 
$ f_h|_{x_{i\pm}}^{\mp} , \, f_h|_{p_{k\pm}}^{\mp} , \,  f_h|_{\mu_{m\pm}}^{\mp} $ are positive. The worst case scenario for positivity is having negative flux terms. In that case,

\begin{eqnarray}
0 & \leq & 
\hat{w}_N f_h|_{x_{i\pm}}^{\mp}  
-  \frac{\Delta t^n}{ \alpha s_1 \Delta x_i} 
\partial_p \varepsilon \, |\mu| f_h|_{x_{i\pm}}^{\mp}
=
f_h|_{x_{i\pm}}^{\mp}
\left(
\hat{w}_N   
-  \frac{\Delta t^n}{ \alpha s_1 \Delta x_i} 
\partial_p \varepsilon \, |\mu| 
\right)
\nonumber\\
0 & \leq & 
\hat{w}_{N}  f_h|_{p_{k\pm}}^{\mp} \,
- \frac{\Delta t^n}{ \alpha s_2 \Delta p_k } 
q |E(x,t) \mu | f_h|_{p_{k\pm}}^{\mp}
=
f_h|_{p_{k\pm}}^{\mp} \,
\left( \hat{w}_{N}
- \frac{\Delta t^n}{ \alpha s_2 \Delta p_k } 
q |E(x,t) \mu | 
\right)
\nonumber\\
0 & \leq &
\hat{w}_{N} f_h|_{\mu_{m\pm}}^{\mp}
-\frac{\Delta t^n (1-\mu_{m\pm}^2) }{ \alpha s_3 p \, \Delta \mu_m } 
q|E(x,t)| f_h|_{\mu_{m\pm}}^{\mp}
=
 f_h|_{\mu_{m\pm}}^{\mp}
\left(
\hat{w}_{N}
-\frac{\Delta t^n (1-\mu_{m\pm}^2) }{ \alpha s_3 p \, \Delta \mu_m } 
q|E(x,t)| 
\right)
\nonumber
\end{eqnarray}

We need then for the worst case scenario that 

\begin{eqnarray}
\hat{w}_N   
& \geq & 
\frac{\Delta t^n}{ \alpha s_1 \Delta x_i} 
\partial_p \varepsilon \, |\mu| 
\nonumber\\
\hat{w}_{N}
& \geq &
 \frac{\Delta t^n}{ \alpha s_2 \Delta p_k } 
q |E(x,t) \mu | 
\nonumber\\
\hat{w}_{N}
& \geq & 
\frac{\Delta t^n (1-\mu_{m\pm}^2) }{ \alpha s_3 p \, \Delta \mu_m } 
q|E(x,t)| \, ,  
\nonumber
\end{eqnarray}

or equivalently,

\begin{eqnarray}
\hat{w}_N   \frac{ \alpha s_1 \Delta x_i}{ 
\partial_p \varepsilon \, |\mu|
}
& \geq & 
 {\Delta t^n} 
\nonumber\\
\hat{w}_{N}
 \frac{ \alpha s_2 \Delta p_k }{ 
q |E(x,t) \mu | 
}
& \geq & {\Delta t^n}
\nonumber\\
\hat{w}_{N}
\frac{ \alpha s_3  \, \Delta \mu_m \, p }{ 
q|E(x,t)| 
(1-\mu_{m\pm}^2)
}
& \geq & \Delta t^n  
\nonumber
\end{eqnarray}

Therefore the CFL conditions imposed to satisfy the positivity of the transport term $I$ are

\begin{eqnarray}
\frac{ \alpha s_1 \hat{w}_N   \Delta x_i}{ 
\max_{\hat{r}}
\partial_p \varepsilon(p_{\hat{r}}) \cdot  \,
\max_{\pm} |\mu_{m\pm}|
}
& \geq & 
 {\Delta t^n} 
\nonumber\\
 \frac{ \alpha s_2 \hat{w}_{N}
  \Delta p_k }{ 
q \max_{\hat{q}} |E(x_{\hat{q}},t)| 
\cdot
\max_{\pm} |\mu_{m\pm} | 
}
& \geq & {\Delta t^n}
\nonumber\\
\frac{ \alpha s_3 \hat{w}_{N} \Delta \mu_m
\cdot p_{k-} \,  }{ 
q
\max_{\hat{q}}
|E(x_{\hat{q}},t)| 
\cdot
\max_{\pm}
(1-\mu_{m\pm}^2)
}
& \geq & \Delta t^n  \, .
\nonumber
\end{eqnarray}



Regarding $II$, there are several ways to guarantee its positivity.\\

One possible way to guarantee its positive is given below,
by separating the gain and the loss part, combining the cell average
with the loss term and deriving a CFL condition related to the collision frequency, and imposing a positivity condition on 
the points where the gain term is evaluated, which differs
for inelastic scatterings from the previous Gauss-Lobatto points because of the addition or subtraction of the phonon energy $\hbar \omega $. We would need an additional set of points in which to impose positivity in order to guarantee positivity of $II$ as a whole, since

\begin{eqnarray}
&&
II  =  \bar{f}_{ikm}^n + \frac{\Gamma_C}{1-\alpha}   =
\nonumber\\
&&
\bar{f}_{ikm}^n
 + 
\left[
{2\pi}  
\int_{\Omega_{ikm}} 
\left( 
\sum_{j=-1}^{+1} c_j \, \chi(\varepsilon(p) + j\hbar\omega)
\int_{-1}^{+1} d\mu' \left. \left[
f_h(x,p(\varepsilon'),\mu') \, p^2(\varepsilon') \frac{dp'}{d\varepsilon'} \right] \right|_{\varepsilon' = \varepsilon(p) + j\hbar\omega }  
\right) 
\, p^2 \, dp d\mu dx  
\right.
\nonumber\\
&&
\left.
-
{4\pi} 
\int_{\Omega_{ikm}} 
f_h(x,p,\mu,t)  
\underbrace{
\left(
\sum_{j=-1}^{+1} c_j \, \chi(\varepsilon(p) - j\hbar\omega)  \left. \left[ p^2(\varepsilon') \frac{dp'}{d\varepsilon'} \right] \right|_{\varepsilon' = \varepsilon(p) - j\hbar\omega } 
\right)
}_{ = \, \nu(p) > 0, \quad \mbox{since} \quad \partial_p\varepsilon > 0, \quad c_j > 0, \quad \chi \geq 0 }
\, p^2 \, dp d\mu dx  
\right]\frac{\Delta t^n}{V_{ikm}(1-\alpha)} =
\nonumber\\
&&
\left[
\frac{2\pi \Delta t^n}{ (1-\alpha)}
\int_{\Omega_{ikm}} 
\left( 
\sum_{j=-1}^{+1} c_j \, \chi(\varepsilon(p) + j\hbar\omega)
\int_{-1}^{+1} d\mu' \left. \left[
f_h(x,p(\varepsilon'),\mu') \, p^2(\varepsilon') \frac{dp'}{d\varepsilon'} \right] \right|_{\varepsilon' = \varepsilon(p) + j\hbar\omega }  
\right) 
\, p^2 \, dp d\mu dx  
\right. +
\nonumber\\
&&
\left.
\int_{\Omega_{ikm}} f_h dV
-
\frac{4\pi \Delta t^n}{ (1-\alpha)}
\int_{\Omega_{ikm}} 
f_h  
\left(
\sum_{j=-1}^{+1} c_j \, \chi(\varepsilon(p) - j\hbar\omega)  \left. \left[ p^2(\varepsilon') \frac{dp'}{d\varepsilon'} \right] \right|_{\varepsilon' = \varepsilon(p) - j\hbar\omega } 
\right)
\, p^2 \, dp d\mu dx  
\right]\frac{1}{V_{ikm}} =
\nonumber\\
&&
\left[
\frac{2\pi \Delta t^n}{ (1-\alpha)}
\sum_{j=-1}^{+1} c_j
\int_{\Omega_{ikm}} 
\int_{-1}^{+1} d\mu' \left. \left[
f_h(x,p(\varepsilon'),\mu') \, p^2(\varepsilon') \frac{dp'}{d\varepsilon'} \right] \right|_{\varepsilon' = \varepsilon(p) + j\hbar\omega }  
\chi(\varepsilon(p) + j\hbar\omega)
\, p^2 \, dp d\mu dx  
\right. +
\nonumber\\
&&
\left.
\int_{\Omega_{ikm}} 
f_h(x,p,\mu,t)  
\left( 1 - \frac{4\pi \Delta t^n}{ (1-\alpha)}
\sum_{j=-1}^{+1} c_j \, \chi(\varepsilon(p) - j\hbar\omega)  \left. \left[ p^2(\varepsilon') \frac{dp'}{d\varepsilon'} \right] \right|_{\varepsilon' = \varepsilon(p) - j\hbar\omega } 
\right)
\, p^2 \, dp d\mu dx  
\right]\frac{1}{V_{ikm}} =
\nonumber\\
&&
\left[
\frac{2\pi \Delta t^n}{ (1-\alpha)}
\sum_{j=-1}^{+1} c_j
|\Omega_{ikm}|
\underbrace{ 
\sum_{s,r,q}  w_{s,r,q} 
f_h(x_s,p'(\varepsilon(p_r) + j\hbar\omega),\mu'_q) \, 
\left[
p'^2(\varepsilon') \frac{dp'}{d\varepsilon'} 
\chi(\varepsilon')
\right] \left\lbrace \varepsilon(p_r) + j\hbar\omega \right\rbrace  
\, p_r^2  
}_{ > 0 \quad \mbox{if} \quad f_h(x_s,p'(\varepsilon(p_r) + j\hbar\omega),\mu'_q) \, > 0 . \quad \mbox{Additional set of points for positivity}  }
\right.
\nonumber\\
&&
+
\left.
\int_{\Omega_{ikm}} 
f_h(x,p,\mu,t)  
\underbrace{
\left( 1 - \frac{4\pi \Delta t^n}{ (1-\alpha)}
\sum_{j=-1}^{+1} c_j \, \chi(\varepsilon(p) - j\hbar\omega)  \left. \left[ p^2(\varepsilon') \frac{dp'}{d\varepsilon'} \right] \right|_{\varepsilon' = \varepsilon(p) - j\hbar\omega } 
\right)
}_{> 0 \, \rightarrow \, (1-\alpha)  \left( max_{GQp} \sum_{j=-1}^{+1} c_j \, \chi(\varepsilon(p) - j\hbar\omega)  \left. \left[ p^2(\varepsilon') \frac{dp'}{d\varepsilon'} \right] \right|_{\varepsilon' = \varepsilon(p) - j\hbar\omega } \right)^{-1} > \Delta t  }
\, p^2 \, dp d\mu dx  
\right]\frac{1}{V_{ikm}} 
\nonumber
\end{eqnarray}

where the notation for the measure of the elements is

\begin{equation}
|\Omega_{ikm}| = \Delta x_i \Delta p_k \Delta \mu_m  \, .
\end{equation}

Another possible way to guarantee positivity for $II$ 
is by considering the collision term as a whole.
The difference between the gain  minus the loss
integrals will give us a smaller source term overall, and therefore a more relaxed CFL condition for $\Delta t^n $. We have that

\begin{eqnarray}
&&
II  =  \bar{f}_{ikm}^n + \frac{\Gamma_C}{1-\alpha}   = \, 
\frac{\int_{\Omega_{ikm}} f_h dV}{V_{ikm}} + 
\frac{\Delta t^n \, \int_{\Omega_{ikm}}Q(f_h)dV}{ (1-\alpha)V_{ikm}}
\,
= 
\,
\frac{\int_{\Omega_{ikm}} f_h dV}{V_{ikm}} +
\nonumber\\
&&
 + 
\left[
{2\pi}  
\int_{\Omega_{ikm}} 
\left( 
\sum_{j=-1}^{+1} c_j \, \chi(\varepsilon(p) + j\hbar\omega)
\int_{-1}^{+1} d\mu' \left. \left[
f_h(x,p(\varepsilon'),\mu') \, p^2(\varepsilon') \frac{dp'}{d\varepsilon'} \right] \right|_{\varepsilon' = \varepsilon(p) + j\hbar\omega }  
\right) 
\, p^2 \, dp d\mu dx  
\right.
\nonumber\\
&&
\left.
-
{4\pi} 
\int_{\Omega_{ikm}} 
f_h(x,p,\mu,t)  
\underbrace{
\left(
\sum_{j=-1}^{+1} c_j \, \chi(\varepsilon(p) - j\hbar\omega)  \left. \left[ p^2(\varepsilon') \frac{dp'}{d\varepsilon'} \right] \right|_{\varepsilon' = \varepsilon(p) - j\hbar\omega } 
\right)
}_{ = \, \nu(p) > 0, \quad \mbox{since} \quad \partial_p\varepsilon > 0, \quad c_j > 0, \quad \chi \geq 0 }
\, p^2 \, dp d\mu dx  
\right]\frac{\Delta t^n}{V_{ikm}(1-\alpha)} =
\nonumber
\end{eqnarray}
\begin{eqnarray}
&&
\left[
\frac{2\pi \Delta t^n}{ (1-\alpha)}
\left\lbrace
\int_{\Omega_{ikm}} 
\left( 
\sum_{j=-1}^{+1} c_j \, \chi(\varepsilon(p) + j\hbar\omega)
\int_{-1}^{+1} d\mu' \left. \left[
f_h(x,p(\varepsilon'),\mu') \, p^2(\varepsilon') \frac{dp'}{d\varepsilon'} \right] \right|_{\varepsilon' = \varepsilon(p) + j\hbar\omega }  
\right) 
\, p^2 \, dp d\mu dx  
\right.
\right. +
\nonumber\\
&&
\left.
- 2 
\left.
\int_{\Omega_{ikm}} 
f_h  
\left(
\sum_{j=-1}^{+1} c_j \, \chi(\varepsilon(p) - j\hbar\omega)  \left. \left[ p^2(\varepsilon') \frac{dp'}{d\varepsilon'} \right] \right|_{\varepsilon' = \varepsilon(p) - j\hbar\omega } 
\right)
\, p^2 \, dp d\mu dx  
\right\rbrace
+ \int_{\Omega_{ikm}} f_h dV
\right]\frac{1}{V_{ikm}} =
\nonumber
\end{eqnarray}
\begin{eqnarray}
&&
\left[
\frac{ \Delta t^n}{ (1-\alpha)}
\int_{\Omega_{ikm}} 
\underbrace{
\left(
2\pi
\sum_{j=-1}^{+1} c_j \, 
\int_{-1}^{+1} d\mu' \left. \left[
f_h(x,p(\varepsilon'),\mu') \, p^2(\varepsilon') \frac{dp'}{d\varepsilon'} \chi(\varepsilon') \right] \right|_{\varepsilon' = \varepsilon(p) + j\hbar\omega }
- f_h \nu(p)   
\right)
}_{Q(f_h)}
\, p^2 \, dp d\mu dx  
\right. 
\nonumber\\
&&
\left. 
+ \int_{\Omega_{ikm}} f_h dV
\right]\frac{1}{V_{ikm}} \, ,
\quad \quad
\nu(p)  =
4 \pi
\sum_{j=-1}^{+1}  c_j \,   \left. \left[ p^2(\varepsilon') \frac{dp'}{d\varepsilon'} \chi(\varepsilon') \right] \right|_{\varepsilon' = \varepsilon(p) - j\hbar\omega } 
= \nu(\varepsilon(p)) \, .
\nonumber
\end{eqnarray}


We will treat then the cell average of the collision term,
the gain minus loss term, as a whole, considering it a source term,
and we will apply the same techniques for positivity preserving DG schemes for transport equations with source terms. We have then that

\begin{eqnarray}
II  &=&  \bar{f}_{ikm}^n + \frac{\Gamma_C}{1-\alpha}   = \, 
\frac{\int_{\Omega_{ikm}} f_h dV}{V_{ikm}} + 
\frac{\Delta t^n \, \int_{\Omega_{ikm}}Q(f_h)dV}{ (1-\alpha)V_{ikm}}
\,
= 
\,
\nonumber\\
&=&   
\frac{1}{V_{ikm}} 
\left[
\int_{\Omega_{ikm}} f_h \, p^2 \, dp d\mu dx
+ 
\frac{\Delta t^n }{ (1-\alpha) }
\int_{\Omega_{ikm}}Q(f_h) \, p^2 \, dp d\mu dx
\right]
\, ,
\nonumber\\
Q(f_h) 
&=& 
2\pi
\sum_{j=-1}^{+1} c_j \, 
\int_{-1}^{+1} d\mu' 
\left. 
f_h(x,p(\varepsilon'),\mu') \, p^2(\varepsilon') \frac{dp'}{d\varepsilon'} \chi(\varepsilon') 
\right|_{\varepsilon' = \varepsilon(p) + j\hbar\omega }
- f_h \nu(p)   
 \, ,
\nonumber\\
\nu(p) 
& = &
4 \pi
\sum_{j=-1}^{+1}  c_j \,   \left. \left[ p^2(\varepsilon') \frac{dp'}{d\varepsilon'} \chi(\varepsilon') \right] \right|_{\varepsilon' = \varepsilon(p) - j\hbar\omega } 
= \nu(\varepsilon(p)) \, .
\end{eqnarray}

We want $II$ to be positive.
If the collision operator part was negative,
we choose the time step $\Delta t^n$ such that $II$ is positive on total. We will get this way our CFL condition in order to guarantee the positivity of $II$. We want that

\begin{eqnarray}
II  &=&  
\frac{1}{V_{ikm}} 
\int_{\Omega_{ikm}} 
\left[
f_h(x,p,\mu ,t) \,
+ 
\frac{\Delta t^n }{ (1-\alpha) }
Q(f_h) (x,p,\mu ,t) 
\right]
\, p^2 \, dp d\mu dx
\, \geq 0 
\nonumber\\
II   &=&  
\frac{|\Omega_{ikm}| }{V_{ikm}} 
\sum_{q,r,s} 
w_q w_r w_s
\left[
f_h(x_q,p_r,\mu_s ,t) \,
+ 
\frac{\Delta t^n }{ (1-\alpha) }
Q(f_h) (x_q,p_r,\mu_s ,t) 
\right]
\, p^2_r 
\, \geq 0 
\nonumber
\end{eqnarray}

If $0 >  Q(f_h) $ for any of the points
$(x_q, p_r, \mu_s)$ at time $t = t^n$, 
then choose $\Delta t^n$ such that

\begin{eqnarray}
0 
& \leq & 
f_h(x_q,p_r,\mu_s ,t) \,
+ 
\frac{\Delta t^n }{ (1-\alpha) }
Q(f_h) (x_q,p_r,\mu_s ,t) 
\nonumber\\
0 
& \leq & 
f_h(x_q,p_r,\mu_s ,t) \,
- 
\frac{\Delta t^n }{ (1-\alpha) }
|Q(f_h)| (x_q,p_r,\mu_s ,t) 
\nonumber\\
\Delta t^n
&\leq &
\frac{ (1 - \alpha) f_h(x_q,p_r,\mu_s ,t)  }{ |Q(f_h)| (x_q,p_r,\mu_s ,t)  }
\nonumber
\end{eqnarray}

Our CFL condition in this case would be then

\begin{eqnarray}
\Delta t^n
&\leq &
(1 - \alpha)
\min_{Q(f_h)(x_q,p_r,\mu_s,t^n) < 0}
\left\lbrace
\frac{  f_h(x_q,p_r,\mu_s ,t^n)  }{ |Q(f_h)| (x_q,p_r,\mu_s ,t^n)  }
\right\rbrace
\end{eqnarray}

The minimum for the CFL condition is taken 
over the subset of Gaussian Quadrature points
$ (x_q,p_r,\mu_s) $
inside the cell $\Omega_{ikm}$ 
(whichever the chosen quadrature rule was)
over which 
$$ Q(f_h)(x_q,p_r,\mu_s, t^n) < 0 .$$ 
This subset of points might be different for each
time $t^n$ then. \\

We have figured out the respective CFL conditions 
for the transport and collision parts. 
Finally, we only need to choose
the optimal parameter $\alpha$ that gives us the most 
relaxed CFL condition for $ \Delta t^n$ such that positivity is preserved for the cell average at the next time, $\bar{f}_{ikm}^{n+1} $. 
The positivity of the whole numerical solution 
to the pdf, not just its cell average, can be guaranteed by applying the limiters in \cite{ZhangShu1}, \cite{ZhangShu2}, which preserve the cell average but modify the slope of the piecewise linear solutions in order to make the function non - negative in case it was negative before.

\section{Stability of the scheme under an entropy norm}

We can prove the stability of the scheme under the entropy norm 
related to the interior product
\begin{equation}
 \int   f_h \, g_h e^H \, p^2 dp d\mu dx \, ,
\end{equation}
inspired in the strategy of Cheng, Gamba, Proft \cite{CGP}.
This estimates are possible due to the dissipative property of the linear collisional operator applied to the curvilinear representation of the momentum, with the entropy norm related to the 
function $e^{H(x,p,t)} = \exp\left(\varepsilon(p) -qV(x,t) \right)$
Assuming periodic boundary conditions in all directions for simplicity of the stability proof, we look for $f_h \, \in \, V_h^k$ such that $\forall \, g_h \, \in V_h^k$ 
and $\, \forall \, \Omega_{ikm} $
\begin{eqnarray} \label{entropyPPDGform}
 \int_{ikm}  \partial_t  f_h \, g_h e^H \, p^2 dp d\mu dx 
&&
\\
-\int_{ikm}  \partial_p \varepsilon (p) \, f_h \, \mu \, \partial_x (g_h e^H) \, p^2 dp d\mu dx \,
&\pm &
 \int_{km} \partial_p \varepsilon \, \widehat{ f_h  \mu }|_{x_{i\pm}}  \, g_h e^H|_{x_{i\pm}}^{\mp}  \, p^2 dp d\mu  
\nonumber\\
- \int_{ikm}
{ p^2 } (-qE)(x,t) f_h \mu \, \partial_p (g_h e^H) \, d\mu dx
& \pm & 
\int_{im} p^2_{k^\pm} \,
(-q  \widehat{ E f_h \mu } )|_{p_{k\pm}} g_h e^H|_{p_{k\pm}}^{\mp} \, d\mu dx
\nonumber\\
- \int_{ikm}
{ (1-\mu^2) f_h }  (-qE)(x,t) \, \partial_{\mu} (g_h e^H) \, p \, dp d\mu dx 
&\pm &
\int_{ik} (1-\mu_{m\pm}^2)
(-q \widehat{ E f_h })|_{\mu_{m\pm}}   \, g_h e^H|_{\mu_{m\pm}}^{\mp} \, p \, dp dx
\nonumber\\
&=&
\int_{ikm} Q(f_h) g_h e^H \, p^2 dp d\mu dx   \, ,
\nonumber
\end{eqnarray}
where we are including  as a factor the inverse of a 
Maxwellian along the characteristic flow generated 
by the Hamiltonian transport field 
$\left(\partial_p \varepsilon(p), q\partial_x V(x,t) \right) $
\begin{equation}
e^{H(x,p,t)} = \exp(\varepsilon(p) -qV(x,t)) = 
\left( e^{qV(x,t)} e^{-\varepsilon(p)} \right)^{-1} \, ,
\end{equation}
which is an exponential of the Hamiltonian energy, 
assuming the energy is measured in $K_B T $ units.

We include this modified inverse Maxwellian factor because 
we can use some entropy inequalities related to the collision operator. 
Our collision operator satisfies the dissipative property
\begin{equation}
\int_{\Omega_{\vec{p}}} Q(f) g d\vec{p} = 
-\frac{1}{2} \int_{\Omega_{\vec{p}}} S(\vec{p}\,'\rightarrow \vec{p})
e^{-\varepsilon(p')} \left(\frac{f'}{e^{-\varepsilon(p')}} -
\frac{f}{e^{-\varepsilon(p)}} \right) (g' - g) d\vec{p}\,' d\vec{p}
\end{equation}
which can be also expressed as (multiplying and dividing by 
$e^{-qV(x,t)}$) 
\begin{equation}
\int_{\Omega_{\vec{p}}} Q(f) g d\vec{p} = 
-\frac{1}{2} \int_{\Omega_{\vec{p}}} S(\vec{p}\,'\rightarrow \vec{p})
e^{-H'} \left(\frac{f'}{e^{-H'}} -
\frac{f}{e^{-H}} \right) (g' - g) d\vec{p}\,' d\vec{p}
\end{equation}

Therefore, if we choose a monotone increasing function 
$g({f}/{e^{-H}})$, namely $g = f/e^{-H} = f e^H $, we have
an equivalent dissipative property but now with
the exponential of the full Hamiltonian
\begin{equation}
\int_{\Omega_{\vec{p}}} Q(f) \frac{f}{e^{-H}} d\vec{p} = 
-\frac{1}{2} \int_{\Omega_{\vec{p}}} S(\vec{p}\,'\rightarrow \vec{p})
e^{-H'} \left(\frac{f'}{e^{-H'}} -
\frac{f}{e^{-H}} \right)^2 d\vec{p}\,' d\vec{p} \leq 0 
\end{equation}
So we have found the dissipative entropy inequality
\begin{equation}
\int_{\Omega_{\vec{p}}} Q(f) fe^H p^2 dp d\mu d\varphi = 
\int_{\Omega_{\vec{p}}} Q(f) \frac{f}{e^{-H}} d\vec{p} 
\leq 0  \, .
\end{equation}

As a consequence of this entropy inequality we obtain the following stability theorem of the scheme under an entropy norm.

\begin{theorem}
\emph{(Stability under the entropy norm $ \int f_h \, g_h e^H \, p^2 dp d\mu dx$):}
Consider the semi-discrete solution $f_h$ 
to the DG formulation in (\ref{entropyPPDGform})
for the BP system in momentum curvilinear coordinates. 
We have then that
\begin{equation}
0 \geq \int_{\Omega}f_h \partial_t  f_h \,  e^{H(x,p,t)} \, p^2 \, dp d\mu dx 
= \frac{1}{2}
\int_{\Omega}  \partial_t  f_h^2 e^{H(x,p,t)} \, p^2 \, dp d\mu dx  \, .
\end{equation}
\end{theorem}

\begin{proof}

Choosing $g_h = f_h$ in (\ref{entropyPPDGform}),
and considering the union of all the cells $\Omega_{ikm}$,
which gives us the 
 whole domain  $\Omega = \Omega_x \times \Omega_{p,\mu}$ for integration, we have
\begin{eqnarray} 
0 \geq 
\int_{\Omega} Q(f_h) f_h e^H \, p^2 dp d\mu dx
&=&
  \int_{\Omega} \partial_t  f_h \, f_h e^H \, p^2 dp d\mu dx 
\nonumber\\
-\int_{\Omega}  \partial_p \varepsilon (p) \, f_h \, \mu \, \partial_x (f_h e^H) \, p^2 dp d\mu dx \,
&+ &
 \int_{\partial_x\Omega} \partial_p \varepsilon \, \widehat{ f_h  \mu }  \, f_h e^H  \, p^2 dp d\mu  
\nonumber\\
- \int_{\Omega}
{ p^2 } (-qE) f_h \mu \, \partial_p (f_h e^H) \, d\mu dx
&+ & 
\int_{\partial_p \Omega} p^2  \,
(-q  \widehat{ E f_h \mu } )  f_h e^H  \, d\mu dx
\nonumber\\
- \int_{\Omega}
{ (1-\mu^2) f_h }  (-qE) \, \partial_{\mu} (f_h e^H) \, p \, dp d\mu dx 
&+ &
\int_{\partial_{\mu} \Omega} (1-\mu^2)
(-q \widehat{ E f_h }) \, f_h e^H \, p \, dp dx
\nonumber
\end{eqnarray}

We can express this in the more compact form
\begin{equation} 
0 \geq \int_{\Omega} \partial_t  f_h \, f_h e^H \, p^2 \, dp d\mu dx 
-\int_{\Omega} f_h \beta \cdot \partial (f_h e^H) \,dp d\mu dx\,
+
\int_{\partial \Omega}
 \widehat{ f_h }  \beta \cdot \hat{n} \, f_h e^H \, d\sigma
\end{equation}
defining the transport vector $\beta$ with the properties
\begin{eqnarray}
\beta &=& \left( p^2 \mu \partial_p \varepsilon(p), -qE \, p^2 \mu, 
-qE p (1-\mu^2) \right) \, , \\
\partial \beta &=&  \partial_{(x,p,\mu)} \beta = 
\left(0, -2pqE\mu, 2\mu qE\right), \quad 
\partial \cdot \beta = -2pqE\mu + 2pqE\mu = 0, \nonumber\\
 \beta \cdot \partial H &=&  \left( p^2 \mu \partial_p \varepsilon(p), -qE \, p^2 \mu, 
-qE p (1-\mu^2) \right) \cdot (qE, \partial_p \varepsilon, 0 ) = 0, \quad \partial_{\mu} H = 0 \, . \nonumber
\end{eqnarray}

We integrate by parts again the transport integrals, obtaining that
\begin{eqnarray}
\int_{\Omega} f_h \beta \cdot \partial (f_h e^H) \,dp d\mu dx\, 
&=&
- \int_{\Omega} \partial \cdot ( f_h \beta) f_h e^H \,dp d\mu dx\, 
+ \int_{\partial \Omega} f_h \beta \cdot \hat{n} f_h e^H \,d\sigma\,
\nonumber\\
&=&
- \int_{\Omega} (\beta \cdot \partial  f_h) f_h e^H \,dp d\mu dx\, 
+ \int_{\partial \Omega} f_h \beta \cdot \hat{n} f_h e^H \,d\sigma\,
\nonumber
\end{eqnarray}
but since
\begin{equation}
\beta \cdot \partial (f_h e^H) = \beta \cdot e^H \partial f_h 
+ \beta \cdot f_h e^H \partial H = e^H \beta \cdot \partial f_h
\end{equation}
 we have then
\begin{equation}
\int_{\Omega} f_h \beta \cdot \partial (f_h e^H) \,dp d\mu dx\, 
=
\int_{\Omega} (\beta \cdot \partial  f_h) f_h e^H \,dp d\mu dx\, 
=
\frac{1}{2} 
\int_{\partial \Omega} f_h \beta \cdot \hat{n} f_h e^H \,d\sigma\, \, .
\end{equation}
We can express our entropy inequality then as
\begin{equation} 
0 \geq \int_{\Omega} \partial_t  f_h \, f_h e^H \, p^2 \, dp d\mu dx 
-\frac{1}{2} 
\int_{\partial \Omega} f_h \beta \cdot \hat{n} f_h e^H \,d\sigma\,
+
\int_{\partial \Omega}
 \widehat{ f_h }  \beta \cdot \hat{n} \, f_h e^H \, d\sigma \, ,
\end{equation}
remembering that we are integrating over the whole domain by considering the union of all the cells defining our mesh.
We distinguish between the boundaries of cells for which $\beta \cdot \hat{n} \geq 0$ and the ones for which $\beta \cdot \hat{n} \leq 0$, defining uniquely the boundaries. Remembering that the upwind flux rule is such that $\hat{f}_h = f_h^-$, we have that the value of the solution inside the cells close to boundaries for which $\beta\cdot\hat{n} \geq 0$ is $f_h^-$, and
for boundaries $\beta \cdot \hat{n} \leq 0 $ the value of the solution inside the cell close to that boundary is $f_h^+$. We have then that
\begin{eqnarray} 
0 &\geq& 
\int_{\Omega} \partial_t  f_h \, f_h e^H \, p^2 \, dp d\mu dx 
-\frac{1}{2} 
\int_{\partial \Omega} f_h \beta \cdot \hat{n} f_h e^H \,d\sigma\,
+
\int_{\partial \Omega}
 f_h^-   \beta \cdot \hat{n} \, f_h e^H \, d\sigma
\nonumber\\
0 &\geq& 
\int_{\Omega} \partial_t  f_h \, f_h e^H \, p^2 \, dp d\mu dx 
-\frac{1}{2} 
\int_{\beta\cdot\hat{n}\geq 0} f_h^- |\beta \cdot \hat{n}| f_h^- e^H \,d\sigma\,
+
\int_{\beta\cdot\hat{n}\geq 0}
 f_h^-  |\beta \cdot \hat{n}| \, f_h^- e^H \, d\sigma
 \nonumber\\
&+& 
\frac{1}{2} 
\int_{\beta\cdot\hat{n}\leq 0} f_h^+ |\beta \cdot \hat{n}| f_h^+ e^H \,d\sigma\,
-
\int_{\beta\cdot\hat{n}\leq 0}
 f_h^-  |\beta \cdot \hat{n}| \, f_h^+ e^H \, d\sigma \, ,
 \nonumber
\end{eqnarray}
and using a notation $e_h$ for the boundaries that allows
redundancy, balanced then by a factor of 1/2, we have 
\begin{eqnarray} 
0 &\geq& 
\int_{\Omega} \partial_t  f_h \, f_h e^H \, p^2 \, dp d\mu dx 
-\frac{1}{2} \left(
\frac{1}{2}
\int_{e_h} f_h^- |\beta \cdot \hat{n}| f_h^- e^H \,d\sigma\,
+
\int_{e_h}
 f_h^-  |\beta \cdot \hat{n}| \, f_h^- e^H \, d\sigma
\right. \nonumber\\
&+& \left.
\frac{1}{2} 
\int_{e_h} f_h^+ |\beta \cdot \hat{n}| f_h^+ e^H \,d\sigma\,
-
\int_{e_h}
 f_h^-  |\beta \cdot \hat{n}| \, f_h^+ e^H \, d\sigma
 \right)
 \nonumber\\
 0 &\geq& 
\int_{\Omega} \partial_t  f_h \, f_h e^H \, p^2 \, dp d\mu dx 
+\frac{1}{2} \left( \frac{1}{2}
\int_{e_h} f_h^- |\beta \cdot \hat{n}| f_h^- e^H \,d\sigma\,
\right. 
 \nonumber\\
&+& \left.
\frac{1}{2} 
\int_{e_h} f_h^+ |\beta \cdot \hat{n}| f_h^+ e^H \,d\sigma\,
-
\int_{e_h}
 f_h^-  |\beta \cdot \hat{n}| \, f_h^+ e^H \, d\sigma
\right)
 \nonumber\\
 0 &\geq& 
\int_{\Omega} \partial_t  f_h \, f_h e^H \, p^2 \, dp d\mu dx 
 \nonumber\\
&+& 
\frac{1}{4} \left(
\int_{e_h} f_h^- f_h^- |\beta \cdot \hat{n}|  e^H \,d\sigma\,
-2
\int_{e_h}
 f_h^- f_h^+  |\beta \cdot \hat{n}| \,  e^H \, d\sigma
+
\int_{e_h} f_h^+ f_h^+ |\beta \cdot \hat{n}|  e^H \,d\sigma\,
\right)
 \nonumber\\
 0 &\geq& 
\int_{\Omega} \partial_t  f_h \, f_h e^H \, p^2 \, dp d\mu dx 
+
\frac{1}{4} 
\int_{e_h} (f_h^+ - f_h^-)^2 |\beta \cdot \hat{n}|  e^H \,d\sigma\, .
\end{eqnarray}

Since the second term is non-negative, we conclude therefore that
\begin{equation}
0 \geq \int_{\Omega}f_h \partial_t  f_h \,  e^{H(x,p,t)} \, p^2 \, dp d\mu dx 
= \frac{1}{2}
\int_{\Omega}  \partial_t  f_h^2 e^{H(x,p,t)} \, p^2 \, dp d\mu dx  \, ,
\end{equation}
and in this sense is that the numerical solution has stability
with respect to the considered entropy norm.
\end{proof}

As a remark, we obtain the corollary

\begin{corollary}
\emph{(Stability under the entropy norm for a time independent Hamiltonian):}
If $V=V(x)$, so $\partial_t H = 0$,
the stability under our entropy norm gives us that for $t\geq 0$
 \begin{equation}
\left| \left| f_h \right| \right|_{L^2_{e^H p^2}}^2 (t)
=
\int_{\Omega}  f_h^2(x,p,\mu,t) e^{H(x,p)} \, p^2 \, dp d\mu dx 
\leq 
\left| \left| f_h \right| \right|_{L^2_{e^H p^2}}^2 (0) \quad .
\end{equation}
\end{corollary}

\begin{proof}
The corollary follows from the fact that, since 
$\partial_t H = -q \partial_t V = 0$, we have
 \begin{equation}
0 \geq 
\int_{\Omega}  \partial_t \left( f_h^2 e^{H(x,p)} \right) \, p^2 \, dp d\mu dx 
= \frac{d}{dt}
\int_{\Omega}  f_h^2(x,p,\mu,t) e^{H(x,p)}  \, p^2 \, dp d\mu dx  \, .
\end{equation}
Since the entropy norm is a decreasing function of time,
our result follows immediately.
\end{proof}

\section{Conclusions}

The work presented here relates to the
development of positivity preserving DG schemes 
for BP semiconductor models. Due to the physics of 
energy transitions given by Planck's law, and to reduce the dimension of the associated collision operator, given its mathematical form,
we pose the Boltzmann Equation for electron transport in
curvilinear coordinates for the momentum. This is a more general form that includes the two other BP models used in the previous lines of research as particular cases. 
We consider the 1D diode problem with azimuthal symmetry assumptions, which give us a 3D plus time problem. 
We choose for this problem the spherical coordinate system $\vec{p}(p,\mu,\varphi)$, slightly different to the previous choices, because its DG formulation gives simpler integrals involving just piecewise polynomial functions for both transport and collision terms. Using the strategy in \cite{ZhangShu1}, \cite{ZhangShu2}, \cite{CGP} we treat the collision operator as a source term, and find convex combinations of the transport and collision terms which guarantee the propagation of positivity of the cell average of our numerical probability density function at the next time step. The positivity of the numerical solution to the pdf in the whole domain can be guaranteed just by applying the limiters in \cite{ZhangShu1}, \cite{ZhangShu2} that preserve the cell average but modify the slope of the piecewise linear solutions in order to make the function non - negative. We have been able to prove as well the stability of the 
semi-discrete DG scheme formulated under an entropy norm, assuming periodic boundary conditions for simplicity. For the simpler case of a time dependent Hamiltonian, the decay of the entropy norm of the numerical solution over time follows as a corollary. This highlights the importance of the dissipative properties of our collisional operator given by its entropy inequalities. In this case, the entropy norm depends on the full time dependent Hamiltonian rather than just the Maxwellian associated solely to the kinetic energy.


\begin{thebibliography}{1}


\bibitem{ZhangShu1}
X.~Zhang and C.-W.~Shu, 
\emph{Positivity-preserving high order discontinuous Galerkin schemes for compressible Euler equations with source terms}, 
J. Comput. Phys., 230 (2011) 1238-1248.

\bibitem{ZhangShu2}
X.~Zhang and C.-W.~Shu, 
\emph{On positivity-preserving high order discontinuous Galerkin schemes for compressible Euler equations on rectangular meshes}, 
J. Comput. Phys., 229 (2010) 8918-8934.

\bibitem{CGP}
Y.~Cheng, I. M. Gamba and J.~Proft, 
\emph{Positivity-preserving discontinuous Galerkin schemes for linear Vlasov-Boltzmann transport 
equations}, 
Mathematics of Computation, 81 (2012) 153-190. 

\bibitem{EECHXM-JCP}
Eirik Endeve, Cory Hauck, Yulong Xing, Anthony Mezzacappa, \emph{Bound-preserving discontinuous Galerkin methods for conservative phase space advection in curvilinear coordinates}, Journal of Computational Physics, 2015.

\bibitem{MP}
A.~Majorana and R.~Pidatella, \emph{A finite difference scheme
solving the Boltzmann Poisson system for semiconductor devices},
Journal of Computational Physics, 174 (2001) 649-668.

\bibitem{carr03} 
J.A.~Carrillo, I.M.~Gamba, A.~Majorana and C.-W.~Shu,
\emph{A WENO-solver for the transients of Boltzmann-Poisson system
for semiconductor devices. Performance and compa\-risons with Monte Carlo methods}, 
Journal of Computational Physics, 184 (2003) 498-525.

%
\bibitem{cgms06}
J.A.~Carrillo, I.M.~Gamba, A.~Majorana and C.-W.~Shu, 
\emph{2D semiconductor device simu\-lations by WENO-Boltzmann schemes:
efficiency, boundary conditions and comparison to Monte Carlo methods}, 
Journal of Computational Physics, 214 (2006) 55-80.
%
\bibitem{CGMS-CMAME2008} 
Y.~Cheng, I. M. Gamba, A.~Majorana and C.-W.~Shu 
\emph{A discontinuous Galerkin solver for Boltzmann-Poisson systems in nano-devices}, 
Computer Methods in Applied Mechanics and Engineering,
198 (2009) 3130-3150.





\end{thebibliography}
\end{document}